\documentclass[]{article}
\usepackage{amsmath,amssymb,amsthm,mathtools}
\usepackage{graphicx}
\usepackage{verbatim}
\usepackage[mathscr]{euscript}

\newtheorem{defn}{Definition}
\newtheorem{thm}{Theorem}

\newtheorem{cor}{Corollary}
\newtheorem{prop}{Proposition}
\newtheorem{lem}{Lemma}
\newtheorem{prob}{Problem}

\providecommand{\keywords}[1]{\text{\textbf{Keywords:}} #1}

\title{Convergence properties of spline-like cardinal interpolation operators acting on $l^p$ data }

\author{Jeff Ledford}
\date{}

\begin{document}
\allowdisplaybreaks

\maketitle

\begin{abstract}
\noindent If $f\in \{f\in L^p(\mathbb{R}): f(x)=\int_{-\pi}^{\pi}e^{ix\xi}d\beta(\xi), \beta\in B.V.([-\pi,\pi])  \}$ , then $f$ is determined by its samples on the integers by taking an appropriate limit.  Specifically, $\| f - L_{\phi_\alpha}f \|_{L^p(\mathbb{R})}\to 0$ as $\alpha\to\infty$ provided that $\{\phi_\alpha: \alpha\in A\}$ is what we call a spline-like family of cardinal interpolators.\\
\end{abstract}

\section{Introduction}

This paper continues the study of convergence properties of cardinal interpolation operators, which was popularized by the spline methods developed by Schoenberg and others.  Ultimately, our goal is to prove a result that is similar to the sampling theorem.  This classical result shows that a Paley-Wiener function is determined by its values at the integers.  Cardinal splines, so called because they interpolate data on the integers, may be used to interpolate $l^2$ data, then by letting the degree increase to infinity, the corresponding Paley-Wiener function is recovered pointwise.  A similar result for a general class of interpolants was proved by the author in \cite{z paper}.  In \cite{Marsden}, Marsden, Richards, and Riemenschneider showed that splines may also be used to recover functions from the space
\begin{equation}\label{A}
\mathcal{A}=\left\{f\in\ L^P(\mathbb{R}): f(x)=(2\pi)^{-1/2}\int_{-\pi}^{\pi} e^{ix\xi} d\beta(\xi), \beta\in B.V.([-\pi,\pi])   \right\}.
\end{equation}
We wish to extend this result to a general class of interpolants, those that are built from what we call spline-like cardinal interpolators.

This remainder of this paper is organized into three sections, the first dealing with definitions and preliminary facts concerning spline-like cardinal interpolators, the second proving the convergence result, and finally in the third section we show that this class of interpolants includes odd-degree cardinal splines, the Gaussian which was the subject of \cite{RS}, Poisson kernels, and two families of multiquadrics.  Of these examples, the convergence property for the Poisson kernels appears to be new, as does the result concerning the family of multiquadrics $\{(x^2+c^2)^{k-1/2}:k\in\mathbb{N}\}$.


\section{Definitions and Basic Facts}

We begin with a convention for the Fourier transform.
\begin{defn}
The \emph{Fourier transform} of a function $f(x)\in L^1(\mathbb{R})$, is defined to be the function
\[
\hat{f}(\xi)=(2\pi)^{-1/2}\int_{\mathbb{R}}f(x)e^{-ix\xi}dx.
\]
\end{defn}
We make the usual extension to the class of tempered distributions using this convention.
Our work will also make use of the mixed Hilbert transform which is described below, more details may be found in \cite{Marsden}.
For $x\in\mathbb{R}$, let $m_x$ be the integer such that $m_x-1/2 \leq x <m_x+1/2$.  If $f=\{f(k)\}\in l^p$ and $1<p<\infty$, we define the \emph{mixed Hilbert transform}, denoted $\mathscr{H}[f](x)$, to be
\[
\mathscr{H}[f](x)=\sum_{k\neq m_x} \dfrac{f(k)}{x-k}.
\]
Proposition 1.3 in \cite{Marsden} provides a constant depending only on $p$ such that
\begin{equation}\label{mH bnd}
\| \mathscr{H}[f]\|_{L^p(\mathbb{R})}\leq C_p \| f \|_{l^p}
\end{equation}

We need a definition similar to the one found in \cite{z paper}. For our purposes, we make the following definition.
\begin{defn}  We say that a function $\phi(x)$ satisfies the \emph{interpolating conditions} if it satisfies the all of the following:
\begin{enumerate}
\item[(A1)]$\phi(x)$ is a real valued slowly increasing function on $\mathbb{R}$,
\item[(A2)] $\hat{{\phi}}(\xi) \geq 0$ and $\hat{{\phi}}(\xi)\geq m >0$ in $[-\pi,\pi]$,
\item[(A3)] $\hat{{\phi}}(\xi)\in C^{1}(\mathbb{R}\setminus\{0\})$, and
\item[(A4)] there exists $\epsilon >0$ such that for $j=0,1$ we have\\ $\hat\phi^{(j)}(\xi)=O(| \xi |^{-(1+\epsilon)})$ as $|\xi|\to\infty$.
\end{enumerate}
\end{defn}
These conditions, assure us that the fundamental function $L_{\phi}(x)$ defined by its Fourier transform via the formula
\begin{equation}\label{L}
 \hat{L}_{\phi}(\xi)=(2\pi)^{-1/2} \dfrac{\hat{\phi}(\xi)}{\displaystyle\sum_{j\in\mathbb{Z}} \hat{\phi}(\xi+2\pi j)  },
\end{equation}
is continuous and solves the following interpolation problem.
\begin{prob}
Find a function  $L\in L^2(\mathbb{R})$ that satisfies $L(j)=\delta_{0,j}$ for $j\in\mathbb{Z}$.
\end{prob}
Functions defined by \eqref{L} are often called \emph{fundamental functions}.  By \emph{interpolator}, we mean the function $\phi$ from which a fundamental function is built.  A fundamental function allows us to solve an $l^2$ interpolation problem using interpolants of the form
\begin{equation}\label{I}
\mathscr{I}_{\phi}[f](x)=\sum_{j\in\mathbb{Z}}f(j)L_{\phi}(x-j),
\end{equation}
where $L_\phi(x)$ is defined in \eqref{L} and $\{f(j)\}\in l^2$.  Henceforth, we focus on interpolants of this form. 
Let us define an auxiliary function that will prove useful.  Given a function $f$ which satisfies (A1)-(A4) and $k\in\mathbb{Z}$, we define
\begin{equation}\label{aux}
\mathscr{M}[f]_{k}(u)=\dfrac{f(u+2\pi k)}{f(u)},\quad |u|\leq \pi.
\end{equation}

Since we are interested in matters of convergence, we need to introduce a parameter.  To this end we let $A\subset (0,\infty)$ be an unbounded index set which could be discrete or continuous depending on the example.  We wish to make sure that our parameter interacts well with the limit.  The following regularity definition makes use of the auxiliary function introduced in \eqref{aux}.
\begin{defn}
A collection of functions $\{\phi_\alpha : \alpha\in A\}$ will be called a \emph{spline-like family of cardinal interpolators} if the following conditions are satisfied:
\begin{enumerate}
\item[(B1)] for all $\alpha\in A$, $\phi_\alpha$ satisfies the interpolating conditions,
\item[(B2)] $\displaystyle \sum_{j\in\mathbb{Z}}\sum_{k\neq j} \| \mathscr{M}[\hat{\phi}_\alpha]_{j}(\xi) \mathscr{M}[\hat{\phi}'_\alpha]_{k}(\xi)  \|_{L^1([-\pi,\pi])} \leq C$, independent of $\alpha$,
\item[(B3)] for $j\neq 0$, $\displaystyle\lim_{\alpha\to\infty}\mathscr{M}[\hat{\phi}_\alpha]_{j}(\xi)=0$ for almost every $|\xi|\leq \pi$, and
\item[(B4)] there exists $\{\mathscr{M}_j\}\in l^1$ such that $\mathscr{M}[\hat{\phi}_\alpha]_{j}(\xi)\leq \mathscr{M}_j$ for $j\neq 0, \alpha\in A$.
\end{enumerate}
\end{defn}

This definition seems quite restrictive, nevertheless many popular choices of interpolator satisfy these conditions.  The reader interested in examples may skip ahead to the final section.

We turn now to preliminary results which will aid in the proof of the main theorem.  Most of these are straightforward consequences of the conditions listed above.  For the remainder of the paper, we fix a particular spline-like family of cardinal interpolators $\{\phi_\alpha: \alpha\in A\}$, and note that the particular value of the constant $C$ depends on its occurrence and may change from one line to the next.

\begin{lem}
If $f=\{f(k)\}\in l^2$, $ L_{\phi_\alpha}(x)  $ is defined by \eqref{L}, and $\mathscr{I}_{\phi_\alpha}[f](x)$ is defined by \eqref{I}, then for each $\alpha\in A$, $L_{\phi_\alpha}\in L^2(\mathbb{R})\cap C(\mathbb{R}) $ and $\mathscr{I}_{\phi_\alpha}[f] \in L^2(\mathbb{R} )\cap C(\mathbb{R})$.
\end{lem}
\begin{proof}
To see that $L_{\phi_\alpha} \in L^2(\mathbb{R})\cap C(\mathbb{R} ) $, we need only consider the Fourier transform.
\begin{align*}
\int_{\mathbb{R}}|\hat{L}_{\phi_\alpha}(\xi)| d\xi & = \sum_{j\in\mathbb{Z}}\int_{-\pi}^{\pi}|\hat{L}_{\phi_\alpha}(\xi-2\pi j)| d\xi\\
& =\int_{-\pi}^{\pi}\sum_{j\in\mathbb{Z}}|\hat{L}_{\phi_\alpha}(\xi-2\pi j)| d\xi =(2\pi)^{1/2}
\end{align*}
Thus $L_{\phi_\alpha} \in C(\mathbb{R})$. Since $\phi_\alpha$ satisfies (A2), the calculation for $L_{\phi_\alpha} \in L^2(\mathbb{R})$ is similar.
To see that $\mathscr{I}_{\phi_\alpha}[f] \in L^2(\mathbb{R})$, we again use the Fourier transform.
\begin{align*}
&\int_{\mathbb{R}}|\widehat{\mathscr{I}_{\phi_\alpha}[f]}(\xi)|d\xi = \int_{\mathbb{R}} \left|\hat{L}_{\phi_\alpha}(\xi) \sum_{k\in\mathbb{Z}}f(k)e^{ik\xi} \right|^2 d\xi \\
= & \sum_{j\in\mathbb{Z}}\int_{-\pi}^{\pi} \left| \hat{L}_{\phi_\alpha}(\xi-2\pi j) \sum_{k\in\mathbb{Z}}f(k)e^{ik\xi}   \right|^2d\xi\\
\leq &\int_{-\pi}^{\pi} \left|  \sum_{k\in\mathbb{Z}}f(k)e^{ik\xi}     \right|^2d\xi = \|\{f(k)\} \|_{l^2}
\end{align*}
The last equality coming from Parseval's formula.  A similar argument shows that $\widehat{\mathscr{I}_{\phi_\alpha}[f]}\in L^1(\mathbb{R})$, hence $\mathscr{I}_{\phi_\alpha}[f]\in C(\mathbb{R})$.
\end{proof}

\begin{lem}
For all $\alpha\in A$, $L_{\phi_\alpha}(x)$ satisfies $L_{\phi_\alpha}(j)=\delta_{0,j}$ for $j\in\mathbb{Z}$.
\end{lem}
\begin{proof}
Using the inversion formula for the Fourier transform, we get for $j\in\mathbb{Z}:$
\begin{align*}
L_{\phi_\alpha}(j) &= (2\pi)^{-1/2}\int_{\mathbb{R}}\hat{L}_{\phi_\alpha}(\xi)e^{ij\xi}d\xi \\
&=(2\pi)^{-1}\sum_{j\in\mathbb{Z}}\int_{-\pi}^{\pi} \hat{L}_{\phi_\alpha}(\xi-2\pi j)e^{ij\xi}d\xi \\
&=(2\pi)^{-1}\int_{-\pi}^{\pi}e^{ij\xi}d\xi = \delta_{0,j}.
\end{align*}
\end{proof}
We introduce an operator which plays an important role throughout the remainder of the paper, the Whittaker map $W:l^p\to L^p(\mathbb{R})$, given by
\[
\mathscr{W}[\{c_k\}](x)=\sum_{k\in\mathbb{Z}}c_k\dfrac{\sin(\pi(x-k))}{\pi(x-k)}.
\]
We mention a few results concerning $\mathscr{W}$. In \cite{Marsden}, it is shown (in Lemma 1.4), that this a continuous mapping for $1<p<\infty$, i.e.
\begin{equation}
\|\mathscr{W}[\{c_k\}] (x) \|_{L^p(\mathbb{R})}\leq C_p\|  \{c_k\}\|_{l^p},
\end{equation}
where $C_p$ depends only on $p$.
Furthermore, it is shown (in Theorem 3.4) that the space 
\[
\mathcal{B}=\left\{ f: f(x)=\mathscr{W}[\{c_k\}](x), \{c_k\}\in l^p \right\}
\]
is equivalent to the space $\mathcal{A}$ defined in \eqref{A}.


\section{Main Result}

Our goal is to prove a generalization of Theorem 3.4 in \cite{Marsden}, which gives another equivalent formulation of the space $\mathcal{A}$. Our argument closely resembles the one given there.  We begin by finding the $L^p$-norm of $L_{\phi_\alpha}$.

\begin{lem}
If $\{\phi_\alpha:\alpha\in A\}$ is a spline-like family of cardinal interpolants, then there exists a constant $C$ independent of $\alpha$ such that
\begin{equation}\label{deriv bnd}
\| D\hat{L}_{\phi_\alpha} \|_{L^1(\mathbb{R})} \leq C
\end{equation}
\end{lem}
\begin{proof}
The quotient rule gives us
\[
D\hat{L}_{\phi_\alpha}(\xi) =(2\pi)^{-1/2} \dfrac{\hat\phi'_\alpha(\xi)P_\alpha(\xi) - \hat\phi_\alpha(\xi)P'(\xi) }   { P^2_\alpha(\xi) },
\]
where $P_\alpha(\xi)=\sum_{j\in\mathbb{Z}}\hat\phi_\alpha(\xi-2\pi j) $.
We have 
\begin{align*}
&(2\pi)^{1/2}\int_{\mathbb{R}}\left|   \dfrac{\hat\phi'_\alpha(\xi)P_\alpha(\xi) - \hat\phi_\alpha(\xi)P'(\xi) }   { P^2_\alpha(\xi) }    \right|d\xi \\
=& \sum_{m\in\mathbb{Z}}\int_{(2m-1)\pi}^{(2m+1)\pi}\left|  \dfrac{\hat\phi'_\alpha(\xi)P_\alpha(\xi) - \hat\phi_\alpha(\xi)P'(\xi) }   { P^2_\alpha(\xi) }    \right|d\xi \\
=& \sum_{m\in\mathbb{Z}}\int_{-\pi}^{\pi} \left|   \dfrac{\hat\phi'_\alpha(\xi+2\pi m)P_\alpha(\xi) - \hat\phi_\alpha(\xi+2\pi m)P'(\xi) }   { P^2_\alpha(\xi) }   \right|d\xi   .
\end{align*}
The integrand may be simplified by letting $u_\alpha(\xi)=\sum_{j\neq 0} \hat{\phi}_{\alpha}(\xi-2\pi j)$ for $|\xi|\leq\pi$.  Then (A3) and (A4) imply that we can differentiate term by term, so we get
\begin{align*}
& \sum_{m\in\mathbb{Z}}\int_{-\pi}^{\pi} \left|   \dfrac{\hat\phi'_\alpha(\xi+2\pi m)P_\alpha(\xi) - \hat\phi_\alpha(\xi+2\pi m)P'(\xi) }   { P^2_\alpha(\xi) }   \right|d\xi \\
=& \sum_{m\in\mathbb{Z}}\int_{-\pi}^{\pi} \left| \dfrac{\hat\phi'_\alpha(\xi+2\pi m)u_\alpha(\xi) - \hat\phi_\alpha(\xi+2\pi m)u_\alpha'(\xi) }   { P^2_\alpha(\xi) }     \right|d\xi\\ \leq & \int_{-\pi}^{\pi} \sum_{m\in\mathbb{Z}}\left| \dfrac{\hat\phi'_\alpha(\xi+2\pi m)u_\alpha(\xi+2\pi m) - \hat\phi_\alpha(\xi+2\pi m)u_\alpha'(\xi+2\pi m) }   { \hat\phi_\alpha^2(\xi) }     \right|d\xi \\
\leq & 2 \int_{-\pi}^{\pi} [\hat\phi_\alpha(\xi)]^{-2}\sum_{m\in\mathbb{Z}}| \hat{\phi}_\alpha(\xi+2\pi m) \sum_{j\neq 0}\hat\phi'_{\alpha}(\xi+2\pi(m-j))     |d\xi \\
\leq & 2\int_{-\pi}^{\pi} \sum_{m\in\mathbb{Z}} \sum_{n\neq m} | \mathscr{M}[\hat{\phi}_\alpha]_{m}(\xi) \mathscr{M}[\hat{\phi}'_\alpha]_{n}(\xi)      |    d\xi \\
=& 2\sum_{m\in\mathbb{Z}}\sum_{n\neq m}\int_{-\pi}^{\pi}  | \mathscr{M}[\hat{\phi}_\alpha]_{m}(\xi) \mathscr{M}[\hat{\phi}'_\alpha]_{n}(\xi) |          d\xi  \leq C
\end{align*}
The constant is independent of $\alpha$ from (B2).  The interchange of the order of the sums and the integral is justified by Tonelli's theorem.
\end{proof}
 We have the following two straightforward corollaries.
 
\begin{cor}
If the conditions of the above lemma are met, then we have the pointwise bound
\begin{equation}\label{pt bnd}
|L_{\phi_\alpha}(x)|\leq C(1+|x|)^{-1},
\end{equation}
where C is independent of $\alpha$.
\end{cor}\label{cor 1}
\begin{proof}
This follows from taking the Fourier transform of $(1+x)L_{\phi_\alpha}(x)$.
\end{proof}

\begin{cor}\label{cor 2}
For $1<p <\infty$, $\|L_{\phi_\alpha}\|_{L^p(\mathbb{R})} \leq C_p$, where $C_p$ is independent of $\alpha$.
\end{cor}
\begin{proof}
This follows immediately from the above corollary.
\end{proof}

We may now establish a bound for $\|\mathscr{I}_{\phi_\alpha}   \|_{l^p\to L^p}$.  
\begin{prop}
Let $\{\phi_\alpha:\alpha\in A\}$ be a spline-like family of cardinal interpolators, $f=\{f(k)\}\in l^p$, and $1< p < \infty$, then
\begin{equation}\label{I bnd}
\|\mathscr{I}_{\phi_\alpha}[f]   \|_{L^p(\mathbb{R})} \leq C_p \|  f \|_{l^p},
\end{equation}
where $\mathscr{I}_{\phi_\alpha}$ is defined as in \eqref{I} and $C_p$ is independent of both $f$ and $\alpha$.
\end{prop}

\begin{proof}
The method of proof is similar to Theorem 3.1 in \cite{Marsden}, we encourage the reader to consult that proof as well.  We will use the converse of H\"{o}lder's inequality.  To this end, we suppose that $1<p<\infty$ and $f=\{f(k)\}\in l^p$ satisfies $f(k)=0$ for $|k|>N$.  Letting $q$ be the conjugate exponent to $p$, i.e. $1/p+1/q=1$, we suppose that $g\in L^q(\mathbb{R})$ is supported in $[-R,R]$.  Recalling that $m_x$ is the integer which satisfies $m_x-1/2\leq x < m_x+1/2$, we have
\begin{align*}
&\left| \int_{\mathbb{R}} \mathscr{I}_{\phi_\alpha} (x) g(x) dx \right| = \left| \int_{-R}^{R}\sum_{|j|\leq N}f(j)L_{\phi_\alpha}(x-j)g(x) dx \right| \\
\leq & \left| \int_{-R}^{R}|f(m_x)|L_{\phi_\alpha}(x-m_x)g(x)dx\right| + \left| \int_{-R}^{R} \sum_{j\neq m_x}f(j)L_{\phi_\alpha}(x-j)g(x)    dx\right| \\
\leq & C_p\| f \|_{l^p}\|g \|_{L^q(\mathbb{R})} + \left| \int_{-R}^{R} \sum_{j\neq m_x}f(j)L_{\phi_\alpha}(x-j)g(x)    dx\right|
\end{align*} 
Here we have used Corollary \ref{cor 2} along with H\"{o}lder's inequality to estimate the first term, we note that the constant $C_p$ is independent of $\alpha$.  It remains to bound the second term.
\begin{align*}
&\left| \int_{-R}^{R} \sum_{j\neq m_x}f(j)L_{\phi_\alpha}(x-j)g(x)    dx\right| \\
= & (2\pi)^{-1/2}\left| \int_{-R}^{R} \sum_{j\neq m_x}f(j)\int_{\mathbb{R}}\hat{L}_{\phi_\alpha}(\xi)e^{i(x-j)\xi}d\xi g(x)     dx\right| \\
= & (2\pi)^{-1/2}\left| \int_{-R}^{R}\sum_{j\neq m_x}\dfrac{f(j)}{x-j}\int_{\mathbb{R}}e^{ix\xi}D\hat{L}_{\phi_\alpha}(\xi)d\xi   g(x) dx \right|\\
\leq & C \int_{-R}^{R}\left| \mathscr{H}[f](x)g(x)\right| dx \leq C^*_p\| f \|_{l^p}\| g \|_{L^q(\mathbb{R})}
\end{align*}
Here we have used the Fourier representation of $L_{\phi_\alpha}$ and integrated by parts, then used the bound \eqref{deriv bnd} along with H\"{o}lder's inequality and \eqref{mH bnd}.  The constant $C^*_p$ is independent of $\alpha$, so combining the two estimates and taking $N,R\to\infty$ completes the proof. 
\end{proof}

\begin{prop}  Let $\{\phi_\alpha:\alpha\in A\}$ be a spline-like family of cardinal interpolators, $f=\{f(j)\}\in l^p$, and $1< p< \infty$, then we have
\begin{equation}\label{Lp lim}
\lim_{\alpha\to\infty}\| \mathscr{I}_{\phi_\alpha}[f](x)- W[f](x)  \|_{L^{p}(\mathbb{R})}=0.
\end{equation}
\end{prop}
\begin{proof}
From Corollary \ref{cor 2}, we may apply the uniform boundedness theorem.  Thus it is enough to check the result on $y^j=\{\delta_{k,j}:k\in\mathbb{Z}\}$.  We have
\begin{align*}
 &\| \mathscr{I}_{\phi_\alpha}[y^j](x)- W[y^j](x)  \|_{L^p(\mathbb{R})} = \| L_{\phi_\alpha}(x-j)-\dfrac{\sin(\pi(x-j))}{\pi(x-j)}    \|_{L^p(\mathbb{R})}\\
 =& \| L_{\phi_\alpha}(x) - \dfrac{\sin(\pi x)}{\pi x}   \|_{L^p(\mathbb{R})}
\end{align*}
Corollary \ref{cor 1} combined with the dominated convergence theorem will finish the proof, provided that we show the pointwise convergence.  We will use the Fourier transform.  We have
\begin{align*}
&|   L_{\phi_\alpha}(x) - \dfrac{\sin(\pi x)}{\pi x}       | = (2\pi)^{-1/2} \left| \int_{\mathbb{R}}\hat{L}_{\phi_\alpha}(\xi)e^{ix\xi}d\xi -  \int_{-\pi}^{\pi} e^{ix\xi}d\xi \right|\\
\leq & (2\pi)^{-1/2}\left( \int_{-\pi}^{\pi} |\hat{L}_{\phi_\alpha}(\xi)-1  |d\xi + \sum_{j\neq 0}\int_{-\pi}^{\pi}| \hat{L}_{\phi_\alpha}(\xi-2\pi j)  |d\xi\right).
\end{align*}
We will estimate the terms in parentheses separately, for the first term we have
\begin{align*}
& \int_{-\pi}^{\pi} |\hat{L}_{\phi_\alpha}(\xi)-1  |d\xi = \int_{-\pi}^{\pi} \left| \dfrac{u_\alpha(\xi)}{\hat\phi_\alpha(\xi)+u_\alpha(\xi)}  \right|d\xi\\
\leq &  \int_{-\pi}^{\pi} \sum_{j\neq 0} \mathscr{M}[\hat\phi_\alpha]_j(\xi) d\xi.
\end{align*}
The function $ u_\alpha(\xi)$ is the same one introduced in the proof of Lemma 3.  Since (B3) and (B4) hold, the dominated convergence theorem shows that this term tends to $0$.  The calculation for the second term is similar:
\begin{align*}
& \sum_{j\neq 0}\int_{-\pi}^{\pi}| \hat{L}_{\phi_\alpha}(\xi-2\pi j)  |d\xi \leq \int_{-\pi}^{\pi} \sum_{j\neq 0} \mathscr{M}[\hat\phi_\alpha]_j(\xi)   d\xi.
\end{align*}
Again, the terms tend to $0$ by the dominated convergence theorem.  Hence we have shown that
\[
\lim_{\alpha\to\infty}\left| L_{\phi_\alpha}(x)  -\dfrac{\sin(\pi x)}{\pi x}    \right| = 0,
\]
which completes the proof.
\end{proof}
These propositions lead us to our main result, stated in the theorem below.
\begin{thm}
Let $1<p<\infty$ and $\{\phi_\alpha:\alpha\in A \}$ be a spline-like family of cardinal interpolators, then the following spaces are equivalent.
\begin{enumerate}
\item[$\mathcal{A}$] = $\left\{ f\in L^p(\mathbb{R}): f(x)=\int_{-\pi}^{\pi}e^{ix\xi}d\beta(\xi),\beta\in B.V.([-\pi,\pi])   \right\}$
\item[$\mathcal{B}$] = $\left\{  f: f(x)=\mathscr{W}[\{c_k\}](x), \{c_k\}\in l^p    \right\}$
\item[$\mathcal{C}$] = $\left\{ f: f(x)= L^p\displaystyle\lim_{\alpha\to\infty} \mathscr{I}_{\phi_\alpha}[\{c_k\}](x), \{c(k)\}\in l^p    \right\}$
\end{enumerate}
\end{thm}
\begin{proof}
That $\mathcal{A}$ and $\mathcal{B}$ are equivalent is shown in Theorem 3.4 in \cite{Marsden}.  The previous proposition shows that $\mathcal{B}$ and $\mathcal{C}$ are equivalent.  
\end{proof}

We have the following corollary concerning the recovery of functions from the space $\mathcal{A}$.
\begin{cor}
If $f\in\mathcal{A}$ and $\{\phi_\alpha: \alpha\in A  \}$ is a spline-like family of cardinal interpolators, then 
\[
\lim_{\alpha\to\infty}\left| f(x)-\mathscr{I}_{\phi_\alpha}[\{f(k)\}](x)    \right| = 0,
\]
and the convergence is uniform.
\end{cor}

\begin{proof}
Theorem 1 implies that $f(x)=\mathscr{W}[\{f(k)\}](x)$, hence the pointwise estimate from Proposition 2 provides the limit, as well as uniformity.  
\end{proof}


\section{Examples}
In this section we provide examples of spline-like families of cardinal interpolators.  Some of these results appear in the literature. Specifically, the spline analog of the $L^p$ result that is obtained here was worked out by Marsden, Richards, and Riemenschneider in \cite{Marsden} and the Gaussian analog was worked out by Riemenschneider and Sivakumar in \cite{RS}.  The work done in \cite{z paper} for the space $L^2$ suggested that an appropriate generalization to $L^p$ should exist. 

The interpolating conditions (A1)-(A4), are generally straightforward exercises which we leave to the reader.  Most of the calculations involved are routine, so we show only a typical calculation.  Checking (B2) is the most tedious calculation and to ease out work we split the double sum into five pieces defined below. 
\begin{enumerate}
\item[$\Lambda_0$] = $\{ (0,n)\in \mathbb{Z}^2 : n\neq 0    \}$
\item[$\Lambda_1$] = $\{ (m,n)\in \mathbb{Z}^2 : m>0 , n>0, n\neq m    \}$
\item[$\Lambda_2$] = $\{ (m,n)\in \mathbb{Z}^2 : m<0 , n>0  \}$
\item[$\Lambda_3$] = $\{ (m,n)\in \mathbb{Z}^2 : m<0 , n<0, n\neq m    \}$
\item[$\Lambda_4$] = $\{ (m,n)\in \mathbb{Z}^2 : m>0 , n<0   \}$
\end{enumerate}
These correspond more or less to the quadrants so that we can break up absolute values easily.  Finally, we will omit constants from the various transforms to avoid clutter and they cancel out in the fundamental function in any case.

\subsection{Odd degree Splines}
We will examine the spline interpolator whose definition may be found in \cite{Marsden}.  Owing to the positivity requirement in (A2), we will only consider odd degree splines,i.e. those which reduce to an odd degree polynomial off of the integers, whose Fourier transforms are given by $\hat\phi_{k}(\xi)= |\xi|^{-2k}$ for $k\in\mathbb{N}$.  The reader is encouraged to examine \cite{Marsden} for more details.
We begin by verifying (B2) on $\Lambda_0$:
\begin{align*}
&\sum_{(0,n)\in\Lambda_0}\int_{-\pi}^{\pi}|\mathscr{M}[\hat{\phi}_k]_{0}(\xi) \mathscr{M}[\hat{\phi}'_k]_{n}(\xi)|d\xi \\
\leq &\sum_{n\neq 0} \int_{-\pi}^{\pi}|\xi|^{2k}(2k)|\xi+2\pi n|^{-2k-1}d\xi\\
\leq & 2 \sum_{n\neq 0}(2|n|-1)^{-2k} \leq  2 \sum_{n\neq 0} (2|n|-1)^{-2}.
\end{align*}
This sum is independent of $k\in\mathbb{N}$, so we move on to $\Lambda_1$.
\begin{align*}
&\sum_{(m,n)\in\Lambda_1}\int_{-\pi}^{\pi}|\mathscr{M}[\hat{\phi}_k]_{m}(\xi) \mathscr{M}[\hat{\phi}'_k]_{n}(\xi)|d\xi \\
\leq & 2\sum_{m>0}\sum_{n>0}(2|m|-1)^{-2k}(2|n|-1)^{-2k}\\
\leq & 2\left( \sum_{m>0}(2|m|-1)^{-2}  \right)^2
\end{align*}
Analogous calculations show (B2) is satisfied for the remaining sums.
To check (B3), we suppose $j\neq 0$ and $|\xi| <\pi$, then we have
$\pi < |2\pi j + \xi|$ so that
\[
\mathscr{M}[\hat{\phi}_k]_{j}(\xi)\leq (2j +\xi/\pi)^{-2k}
\]
Since the term on the right tends to 0, (B3) is satisfied.
Finally, we see that (B4) is satisfied since if $j\neq 0$, we have
\[
\mathscr{M}[\hat{\phi}_k]_{j}(\xi) \leq (2|j|-1)^{-2}.
\]

Before we move to our next example a word on necessity is in order.  Our conditions (B1)-(B4) cannot be necessary since even order splines do not meet the requirements, nevertheless, convergence is possible as shown in \cite{Marsden}.  The condition that \emph{is} satisfied is that the sum in the denominator of $\hat L$ is bounded away from $0$, which is the purpose of (A2).

\subsection{Gaussian kernels}

We now consider $\{ \exp(-x^2/(4\alpha))   :\alpha \geq 1\}$, which is the focus of \cite{RS}, although we have changed the parameter to suit our purposes.  We have $\hat\phi_{\alpha}(\xi) = \exp(-\alpha \xi^2)  $, so that (B1) is clear.  As for (B2), we again only show $\Lambda_0$ and $\Lambda_1$ since the other cases are similar.
\begin{align*}
&\sum_{(0,n)\in\Lambda_0}\int_{-\pi}^{\pi}|\mathscr{M}[\hat\phi_\alpha]_0(\xi)\mathscr{M}[\hat\phi'_\alpha]_n(\xi)|d\xi\\
\leq & \sum_{n \neq 0} \int_{-\pi}^{\pi} 2\alpha \exp(\alpha\pi^2)  |\xi+2\pi n|\exp(-\alpha(\xi+2\pi n)^2)d\xi \\
\leq & 2 \sum_{n\neq 0} \exp(-4\alpha\pi(|n|^2-|n|))\leq 2 \sum_{n\neq 0} \exp(-4\pi(|n|^2-|n|))
\end{align*}
\begin{align*}
&\sum_{(m,n)\in\Lambda_1}\int_{-\pi}^{\pi}|\mathscr{M}[\hat\phi_\alpha]_m(\xi)\mathscr{M}[\hat\phi'_\alpha]_n(\xi)|d\xi\\
\leq & 2 \sum_{m,n>0}\exp(-4\alpha\pi(m^2-m))\exp(-4\alpha\pi(n^2-n)) \\
\leq & 2\left( \sum_{m>0}\exp(-4\pi(m^2-m))  \right)^2
\end{align*}
Conditions (B3) and (B4) follow from similar reasoning as used for odd-degree splines.

\subsection{Poisson kernels}
We now consider $\{(x^2+\alpha^2)^{-1}:\alpha\geq 1  \}$, the family of Poisson kernels, whose Fourier transforms are given by $\hat\phi_{\alpha}(\xi)= \exp(-\alpha|\xi|)$. 
Since (B1) is clear, we check (B2) for $\Lambda_0$ and $\Lambda_1$.  
\begin{align*}
&\sum_{(0,n)\in\Lambda_0}\int_{-\pi}^{\pi}|\mathscr{M}[\hat\phi_\alpha]_0(\xi)\mathscr{M}[\hat\phi'_\alpha]_n(\xi)|d\xi\\
\leq & \sum_{n \neq 0} \int_{-\pi}^{\pi} \alpha\exp(\alpha\pi)\exp(-\alpha|\xi+2\pi n|)d\xi \\
\leq & 2\sum_{n\neq 0} \exp(-2\pi\alpha(|n|-1)) \leq  4+4\sum_{n> 0} \exp(-2\pi n)
\end{align*}
\begin{align*}
&\sum_{(m,n)\in\Lambda_1}\int_{-\pi}^{\pi}|\mathscr{M}[\hat\phi_\alpha]_m(\xi)\mathscr{M}[\hat\phi'_\alpha]_n(\xi)|d\xi\\
\leq & \sum_{m,n>0}\int_{-\pi}^{\pi} \alpha\exp(2\alpha\pi)\exp{-\alpha(\xi+2\pi m)}\exp(-\alpha(\xi_2\pi n))      d\xi \\
\leq & 2 \sum_{m,n>0} \exp(-2\alpha\pi(m-1))  \exp(-2\alpha\pi(n-1))\\
\leq & 2\left( 1+ \sum_{m>0} \exp(-2\pi m)     \right)^2
\end{align*}
To check (B3) we note that for $|\xi|<\pi$ and $j\neq 0$, $\pi <|2\pi j+\xi|$ so we have
\[
\mathscr{M}[\hat{\phi}_\alpha]_{j}(\xi)\leq \exp(\alpha(\pi-|2\pi j + \xi|)).
\]
Since the exponent is negative, the right hand side tends to 0 as required.
For (B4) we have
\[
\mathscr{M}[\hat{\phi}_\alpha]_{j}(\xi)\leq \exp(-2\pi (|j|-1)).
\]

\subsection{Multiquadrics I} 
We consider the family of multiquadrics $\{ (x^2+c^2)^{k-1/2} : k\in\mathbb{N}   \}$ for a fixed $c>0$.  These kernels correspond to `smoothed out' odd degree splines and the convergence result for these interpolation operators appears to be new.  We need the Fourier transforms which may be found in \cite{Jones}; neglecting constants we have $\hat\phi_k(\xi)= |\xi|^{-k}K_{k}(c|\xi|) $, where 
\[
K_a (u) = \int_{0}^{\infty} \exp(-u\cosh(t))\cosh(a t)dt, \quad u>0
\]
is the Macdonald function, also known as the modified Bessel function of the second kind.  We will make use of the following estimates, which we state as a lemma.
\begin{lem}
For $|a|\geq 1/2$ and $u>0$, we have the inequalities
\begin{equation}\label{K bnd}
(\pi/2)^{1/2}u^{-1/2}\exp(-u) \leq K_a(u)\leq (2\pi)^{1/2}u^{-1/2}\exp(-u)\exp(a^2/(2u)).
\end{equation}
\end{lem}
\begin{proof}
This is just a combination of Corollary 5.12 and Lemma 5.13 in \cite{Wendland}.  These bounds also appeared in \cite{AS} as asymptotic relations. 
\end{proof}
The interpolatory conditions follow easily from the definition of the Macdonald function, also of interest is that $\hat\phi_k(\xi)$ decreases on $(0,\infty)$.
It is worthwhile to start with (B4), as we will see that the cases $k=1$ and $k>1$ will require different arguments.  Let us start with $k=1$ and use \eqref{K bnd}.
\begin{align*}
 \mathscr{M}[\hat{\phi}_1]_{j}(\xi) &\leq 2\exp(1/(2\pi)) \exp(c(\pi-|\xi+2\pi j|))\left( \pi/|\xi+2\pi j|  \right)^{3/2}\\
&\leq  3(2|j|-1)^{-3/2}
\end{align*}
For $k>1$ the estimate is much simpler, we only need to use that the Macdonald function is decreasing.
\begin{align*}
& \mathscr{M}[\hat{\phi}_k]_{j}(\xi) \leq (2|j|-1)^{-k} 
\end{align*} 
Thus we see that (B4) is satisfied with $\mathscr{M}_j=3(2|j|-1)^{-3/2}$.
The above estimates imply that (B3) holds as well.  We now must check (B2).  Again we split up our estimate into $k=1$ and $k>1$.  For $k=1$, our estimate of the $\Lambda_0$ sum is similar to the calculation for (B2), indeed we have 
\begin{align*}
& \sum_{(0,n)\in\Lambda_0}\int_{-\pi}^{\pi} |\mathscr{M}[\hat{\phi}_1]_{0}(\xi)\mathscr{M}[\hat{\phi}'_1]_{n}(\xi)|d\xi \\
\leq & \sum_{n\neq 0} 4 \exp(1/(2\pi)) \exp(-2c\pi(|n|-1))(2|n|-1)^{-3/2}\\
 \leq & 6\sum_{n\neq 0} (2|n|-1)^{-3/2}
\end{align*}
As for $k>1$, we use the corresponding estimate and see that
\begin{align*}
& \sum_{(0,n)\in\Lambda_0}\int_{-\pi}^{\pi}|\mathscr{M}[\hat{\phi}_k]_{0}(\xi)\mathscr{M}[\hat{\phi}'_k]_{n}(\xi)|d\xi \\
\leq & 2\sum_{n\neq 0} (2|n|-1)^{-k} \leq 2\sum_{n\neq 0} (2|n|-1)^{-2}
\end{align*}
So we can see that the sum over $\Lambda_0$ is bounded independent of $k$.  For the sum over $\Lambda_1$ and $k=1$, we have
\begin{align*}
&\sum_{(m,n)\in\Lambda_1}\int_{-\pi}^{\pi}|\mathscr{M}[\hat\phi_1]_m(\xi)\mathscr{M}[\hat\phi'_1]_n(\xi)|d\xi\\
\leq & 18 \left( \sum_{m>0}(2m-1)^{-3/2} \right)^2,
\end{align*}
and for $k>1$ we have
\begin{align*}
&\sum_{(m,n)\in\Lambda_1}\int_{-\pi}^{\pi}|\mathscr{M}[\hat\phi_k]_m(\xi)\mathscr{M}[\hat\phi'_k]_n(\xi)|d\xi \\
\leq & 2\left( \sum_{m>0}(2m-1)^{-2k}   \right)^2 \leq  2\left( \sum_{m>0}(2m-1)^{-2}   \right)^2.
\end{align*}
Hence the sum of $\Lambda_1$ is also bounded independent of $k$.  Analogous bounds hold for the other sums.

If $\{a_k:k\in\mathbb{N}\}\subseteq [1/2,\infty)\setminus\mathbb{N}$ satisfies dist$(\{a_k\}, \mathbb{N})>0$, then a slight modification of the argument given above will show that for $c>0$ fixed, $\{(x^2+c^2)^{a_k}:k\in\mathbb{N}\}$ is a spline-like family of cardinal interpolators as well.

\subsection{Multiquadrics II}
We will now consider the family of multiquadrics $\{(x^2+c^2)^{a}:c\geq 1\}$, where $a\in\mathbb{R}$ is fixed.  The case that $a=1/2$ is the subject of \cite{Baxter}.  This time the Fourier transforms are given by constant multiples of $\hat{\phi}_c(\xi)= |\xi|^{-(a+1/2)}K_{a+1/2}(c|\xi|)$ where $a\notin \tilde{N}=\mathbb{N}\cup\{0\}\cup\{-k-1/2:k\in\mathbb{N}\}$.  In order to use \eqref{K bnd}, we continue with $a\in(\mathbb{R}\setminus\tilde{N})\cap \{a\in\mathbb{R}:|a+1/2|\geq 1/2\}$.  As usual, we check (B2), (B3), and (B4). We check (B4) first.
\begin{align*}
 \mathscr{M}[\hat{\phi}_c]_{j}(\xi) &\leq 2\exp(1/(2\pi)) \exp(c(\pi-|\xi+2\pi j|))\left( \pi/|\xi+2\pi j|  \right)^{a+1/2}\\
&\leq  3(2|j|+1)^{|a|+1/2}\exp(-2c\pi(|j|-1))\\
&\leq 3(2|j|+1)^{|a|+1/2}\exp(-2\pi(|j|-1))
\end{align*} 
For $|\xi|<\pi$ and $j\neq 0$, we can see that $\pi - |\xi+2\pi j|<0$, hence the exponent in the above calculation is negative and all of the terms tend to 0 as $c\to\infty$, which means (B3) is satisfied.  We must now check (B2).  For the sum over $\Lambda_0$, we have
\begin{align*}
& \sum_{(0,n)\in\Lambda_0}\int_{-\pi}^{\pi}|\mathscr{M}[\hat{\phi}_c]_{0}(\xi) \mathscr{M}[\hat{\phi}'_c]_{n}(\xi) |d\xi \\
\leq & 4\exp(a^2/(2\pi))\sum_{n\neq 0} (2|n|+1)^{a+1/2}\exp(-2c\pi(|n|-1))\\
\leq & 4\exp(a^2/(2\pi))\sum_{n\neq 0} (2|n|+1)^{a+1/2}\exp(-2\pi(|n|-1)),
\end{align*}
which is bounded uniformly in $c$.  For $\Lambda_1$, we have
\begin{align*}
& \sum_{(m,n)\in\Lambda_1}\int_{-\pi}^{\pi}| \mathscr{M}[\hat{\phi}_c]_{m}(\xi)\mathscr{M}[\hat{\phi}'_c]_{n}(\xi)  |d\xi \\
\leq & 8\exp(a^2/\pi)\left( \sum_{m>0} (2j+1)^{a+1/2}\exp(-2c\pi(m-1))  \right)^2 \\
\leq & 8\exp(a^2/\pi)  \left( \sum_{m>0} (2j+1)^{a+1/2}\exp(-2\pi(m-1))    \right)^{2}
\end{align*}
Since analogous estimates hold for the other sums, this collection of multiquadrics is also a spline-like family of interpolators.


\end{document}